\documentclass[11pt,a4paper,draft]{amsart}
\usepackage[a4paper,twoside,centering]{geometry}
\usepackage[arrow,matrix]{xy}
\usepackage[shortlabels]{enumitem}

\title[Coefficients of Coxeter polynomials of trees and bipartite quivers]
{On the coefficients of Coxeter polynomials of trees and
bipartite quivers}
\author{Niv Harel \and Sefi Ladkani}

\address{Department of Mathematics, University of Haifa,
199 Abba Khoushy Avenue, Mount Carmel, POB 3338, Haifa 3103301, Israel}
\email{ladkani.math@gmail.com}

\thanks{%
This work was carried out as an undergraduate research project
by the first-named author, under the supervision of the second-named author,
and was supported by the Center for Mathematics and Scientific Computation
at the University of Haifa.
}

\newcommand{\bZ}{\mathbb{Z}}
\newcommand{\wt}{\widetilde}
\DeclareMathOperator{\tr}{tr}

\theoremstyle{plain}
\newtheorem{theorem}{Theorem}[section]
\newtheorem*{theorem*}{Theorem}
\newtheorem{prop}[theorem]{Proposition}
\newtheorem{lemma}[theorem]{Lemma}
\newtheorem{cor}[theorem]{Corollary}
\newtheorem*{cor*}{Corollary}

\theoremstyle{definition}
\newtheorem{defn}[theorem]{Definition}
\newtheorem{remark}[theorem]{Remark}

\numberwithin{equation}{section}

\begin{document}

\begin{abstract}
We apply spectral graph theory and a theorem of A'Campo to express the
first and second coefficients of the Coxeter polynomials associated with
certain bipartite quivers in terms of the degrees of the vertices in
their underlying graphs.
As a consequence, we provide a new proof of a result by Happel, expressing
the second coefficient of the Coxeter polynomial of a tree in terms of its
vertex degrees.
\end{abstract}

\maketitle

\section{Introduction}

The Coxeter polynomial of an acyclic quiver is defined as the characteristic
polynomial of its Coxeter transformation, and its spectral properties
encode various representation-theoretic features of the associated path
algebra, see~\cite{dlPena94,dlPenaTakane90,Ringel94} and the
survey~\cite{LP08}.
Coxeter polynomials have been explicitly computed for various classes of
quivers and, in the case of trees, can be determined recursively using
the algorithm described in~\cite{Boldt95}.

It is an interesting problem to interpret specific coefficients of the
Coxeter polynomial in terms of the underlying combinatorics of the quiver.
As shown in~\cite{Happel97}, the first non-trivial coefficient admits an
interpretation via the Hochschild cohomology of the path algebra, which was
explicitly computed in~\cite[\S1.6]{Happel89}, thereby allowing a description
in terms of the numbers of paths running parallel to arrows,
see~\cite[Remark 4.7]{Happel09}.
In the same paper~\cite{Happel09}, Happel gave a description of the second 
coefficient of the Coxeter polynomial of a tree in terms of its vertex
degrees. His proof relies on a formula for the Coxeter polynomial of a
one-point extension, derived using techniques from the representation theory
of quivers and finite-dimensional algebras.

In this note we derive and present formulae for the first two coefficients
of the Coxeter polynomial of a bipartite quiver without parallel arrows.
Such a quiver on vertex set $\{1, 2, \dots, n\}$
is a directed graph that can be described as follows:
Let $0 \leq m \leq n$, and for each $1 \leq i \leq m$,
let $B_i \subseteq \{m+1, \dots, n\}$. Then, there is an
arrow $i \to j$ if and only if $j \in B_i$.

Let
\begin{align*}
e = \sum_{i=1}^{m} |B_i| &,&
q = \sum_{1 \leq i < j \leq m} \binom{|B_i \cap B_j|}{2} .
\end{align*}
Then $e$ is the number of arrows and $q$ counts the number of occurrences
of the quiver $\wt{A}_{2,2}$, formed by assigning a bipartite
orientation to the graph $K_{2,2}$ (see Figure~\ref{fig:K22}),
as a full subquiver.
In particular, $q=0$ if and only if $|B_i \cap B_j| \leq 1$ for any
$i \neq j$.
Using this notation, the degree $d_i$ of vertex $i$ can be written as
\[
d_i = \begin{cases}
|B_i| & \text{if $i \leq m$}, \\
\left| j \leq m \,:\, i \in B_j \right| & \text{otherwise.}
\end{cases}
\]

\begin{figure}
\[
\xymatrix{
{\bullet} \ar[d] \ar[dr] & {\bullet} \ar[d] \ar[dl] \\
{\bullet} & {\bullet}
}
\]
\caption{The bipartite quiver $\wt{A}_{2,2}$}
\label{fig:K22}
\end{figure}
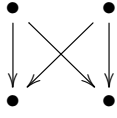

\begin{theorem} \label{t:Q}
Let $Q$ be a bipartite quiver on $n \geq 2$ vertices without parallel arrows,
and let
\[
\phi_Q(x) = x^n + a_1 x^{n-1} + a_2 x^{n-2} + \dots 
\]
be its Coxeter polynomial.
Then $a_1 = n - e$ and
\begin{equation} \label{e:Q:a2}
a_2 = \frac{(n-e)(n-1-e)}{2} + e - 2q - \sum_{i=1}^{n} \binom{d_i}{2} .
\end{equation}
\end{theorem}
The expression in~\eqref{e:Q:a2} can also be rewritten as
\begin{equation} \label{e:Q:a2v}
a_2 = \frac{(n-e)(n+1-e)}{2} - v - 2q - \sum_i \binom{d_i-1}{2} ,
\end{equation}
where $v$ is the number of isolated vertices in $Q$ and the sum
runs over all the non-isolated vertices.

\begin{cor}
If $Q$ is a bipartite quiver without parallel arrows that does not
contain $\wt{A}_{2,2}$ as a full subquiver,
then the coefficients $a_1$ and $a_2$ of its Coxeter polynomial
depend only on the degrees of the vertices in its underlying graph.
\end{cor}

Our proof is completely elementary and avoids the use of
the representation theory of quivers and their path algebras.
Under our hypotheses, the underlying graph $G$ of $Q$ is
a simple bipartite graph.
By A'Campo theorem~\cite{ACampo76}, we obtain an explicit relation
between the Coxeter polynomial $\phi_Q(x)$
and the characteristic polynomial $p_G(x)$ of the
adjacency matrix of $G$.
Spectral graph theory combined with Newton's identities then expresses
the coefficients of $p_G(x)$ in terms of counts of closed
walks in $G$, see for instance~\cite{CRS10}.
In particular, the first nontrivial coefficients depend
solely on the vertex degrees and the number of cycles of length $4$.

We recover Happel's result as a special case.
 
\begin{cor}[\protect{\cite[Theorem~4.8]{Happel09}}] \label{c:T}
Let $T$ be a tree on $n \geq 2$ vertices and let
\[
\phi_T(x) = x^n + a_1 x^{n-1} + a_2 x^{n-2} + \dots 
\]
be its Coxeter polynomial. Then $a_1 = 1$ and
\[
a_2 = 1 - \sum_{i=1}^{n} \binom{d_i-1}{2} .
\]
Hence $a_2 \leq 1$, and equality holds if and only if $T$ is the
Dynkin diagram $A_n$.
\end{cor}

The paper is structured as follows.
In Section~\ref{sec:represent} we explore the notion of representable
polynomials, introduced by Lenzing and de la Pe\~{n}a as a framework for
understanding A'Campo's result.
Section~\ref{sec:Coxeter} reviews the definition of the Coxeter polynomial
associated with an invertible matrix $C$, and shows that when $C=I-N$ with
$N^2=0$, the Coxeter polynomial is representable by the characteristic
polynomial of $N+N^T$, recovering A'Campo's result as a special case.
In Section~\ref{sec:graph} we discuss graphs, quivers and their associated
polynomials. In particular, we express the coefficients $c_1, c_2$ and
$c_4$ of the characteristic polynomial of a simple graph in terms of its
vertex degrees and the number of $4$-cycles.
Finally, Section~\ref{sec:proofs} assembles all ingredients to preset the
proofs of the main results.

\section{Representability of polynomials}
\label{sec:represent}

The notion of representability of polynomials was defined and studied by
Lenzing and de la Pe\~{n}a, see~\cite[\S3.2]{LP08} and~\cite{LP09}.
It provides a framework for A'Campo's result~\cite{ACampo76} relating the
Coxeter polynomial of a bipartite quiver to the characteristic polynomial
of its underlying graph.
In this section we use linear algebra to clarify and explore this concept.

Let $F$ be a field and let $V_n$ be the space of polynomials
of degree at most $n$ with coefficients in $F$.
\begin{defn}
A polynomial $p \in V_n$ is \emph{represented} by a polynomial
$q \in V_n$ if
\[
p(x^2) = x^n q(x + x^{-1}) .
\]
\end{defn}

\subsection{Linear subspaces of polynomials}

\begin{defn}
Define subspaces $W_n$ and $U_n$ of $V_n$ as follows.
\begin{align*}
W_n &= \left \{ c_0 x^n + c_1 x^{n-1} + \dots + c_n \in V_n
\,:\, 
\text{$c_i = c_{n-i}$ for any $0 \leq i \leq n$} \right\} ,
\\
U_n &= \left \{ c_0 x^n + c_1 x^{n-1} + \dots + c_n \in V_n
\,:\, 
\text{$c_i = 0$ for any odd $i$} \right\} . 
\end{align*}
Note that a polynomial $p \in V_n$ belongs to $W_n$ if and only if
$p(x) = x^n p(x^{-1})$.
\end{defn}

\begin{remark}
The elements of $W_n$ are sometimes called \emph{palindromic} 
(or \emph{self-reciprocal}), but note that this notion depends on the
ambient space $V_n$ that we work in. For instance,
while the polynomial $x$ is palindromic as an element in $V_2$, it is
not palindromic as an element in $V_1$.
\end{remark}

\begin{lemma} \label{l:dim}
$\dim W_n = \dim U_n = \dim (W_{2n} \cap U_{2n}) = 1 + \lfloor n/2 \rfloor$.
\end{lemma}
\begin{proof}
We compute the dimension of the space $W_{2n} \cap U_{2n}$, as the other
assertions are clear.
A polynomial in $V_{2n}$ belongs to this space if and only if its
coefficients satisfy $c_i=0$ for odd $i$ and $c_{2n-i}=c_i$ for any
$0 \leq i \leq 2n$. Hence the coefficients are determined by the
sequence $c_0, c_2, \dots, c_{2 \lfloor n/2 \rfloor}$.
\end{proof}

Given a polynomial $p \in V_n$, define new polynomials $S(p)$ and $T(p)$
by
\begin{align*}
S(p)(x) &= p(x^2) , \\
T(p)(x) &= x^n p(x + x^{-1}) .
\end{align*}

\begin{lemma} \label{l:ST:inc}
There are inclusions
\begin{align*}
S(V_n) \subseteq U_{2n} &,& S(W_n) \subseteq W_{2n} &,&
T(V_n) \subseteq W_{2n} &,& T(U_n) \subseteq U_{2n} .
\end{align*}
\end{lemma}
\begin{proof}
The first two inclusions are clear.

If $p \in V_n$, then $T(p)$ is a polynomial of degree at most
$2n$, and
\[
x^{2n}T(p)(x^{-1}) = x^{2n} x^{-n} p(x^{-1} + x) = T(p)(x) ,
\]
so $T(p) \in W_{2n}$.

For the last inclusion,
take a basis element $q(x) = x^{n-2j}$ for some
$j \leq \lfloor n/2
\rfloor$. Then
$T(q)(x) = x^n (x+x^{-1})^{n-2j} = x^{2j} (x^2 + 1)^{n-2j}$,
hence all the coefficients of odd powers vanish.
\end{proof}

\begin{lemma} \label{l:ST:iso}
There are well-defined linear isomorphisms
\begin{align*}
S \colon V_n \to U_{2n} &,&
S|_{W_n} \colon W_n \to W_{2n} \cap U_{2n} &,&
T \colon V_n \to W_{2n} &,&
T|_{U_n} \colon U_n \to W_{2n} \cap U_{2n} .
\end{align*}
\end{lemma}
\begin{proof}
By Lemma~\ref{l:ST:inc}, $S(W_n) \subseteq W_{2n} \cap U_{2n}$ and
$T(U_n) \subseteq W_{2n} \cap U_{2n}$. Hence, in each of the four cases,
the map is one-to-one and sends the source into the target,
which has the same dimension by Lemma~\ref{l:dim}.
\end{proof}

\begin{prop} \label{p:represent}
A polynomial $p \in V_n$ is represented by a polynomial
$q \in V_n$ if and only if $p \in W_n$. In this case, the representing
polynomial $q$ is unique and moreover $q \in U_n$.
\end{prop}
\begin{proof}
A polynomial $p \in V_n$ is represented by $q \in V_n$ if and only if
$S(p) = T(q)$. Now  $S(p) \in U_{2n}$ and $T(q) \in W_{2n}$ by
Lemma~\ref{l:ST:inc}.
As these elements are equal, they must belong to the intersection
$W_{2n} \cap U_{2n}$, so $p \in W_n$ and $q \in U_n$ by
Lemma~\ref{l:ST:iso}, as desired.
If $p \in W_n$, then $q = (T^{-1} S)(p)$ is the unique polynomial
representing it.
\end{proof}

\subsection{Relating the coefficients}

While Proposition~\ref{p:represent} establishes the existence and
uniqueness of the representing polynomial, subsequent calculations
require an explicit relation between the coefficients of a polynomial
and those of its representation.

Let $p, q \in V_n$ and write them as
$p(x) = \sum_{k=0}^{n} a_k x^{n-k}$ and
$q(x) = \sum_{j=0}^{n} c_j x^{n-j}$.

\begin{prop} \label{p:a:c}
If $p$ is represented by $q$, then $c_j = 0$ for any odd $j$, and
\begin{equation} \label{e:a:c}
a_k = a_{n-k} = \sum_{j=0}^{\min(k, n-k)} \binom{n-2j}{k-j} c_{2j}
\end{equation}
for any $0 \leq k \leq n$.
\end{prop}
\begin{proof}
Proposition~\ref{p:represent} implies that $q \in U_n$ and
$p \in W_n$, hence $c_j=0$ for odd $j$ and $a_k = a_{n-k}$ for
any $k$.
Since the maps $S$ and $T$ are linear, we can determine the
contribution of the coefficient $c_{2j}$ to $a_k$ by computing the
coefficient of $x^{2(n-k)}$ in $T(x^{n-2j})$.

As in the proof of Lemma~\ref{l:ST:inc},
\[
T(x^{n-2j}) = x^{2j}(x^2+1)^{n-2j}
= \sum_{i=0}^{n-2j} \binom{n-2j}{i} x^{2j+2(n-2j-i)}
= \sum_{k=j}^{n-j} \binom{n-2j}{k-j} x^{2(n-k)} ,
\]
hence
\[
a_k = \sum_{j=0}^{\min(k,n-k)} \binom{n-2j}{k-j} c_{2j} .
\]
\end{proof}

Let us write explicit expressions for the first coefficients:
\begin{align*}
a_0 &= c_0 \\
a_1 &= n c_0 + c_2 \\
a_2 &= \binom{n}{2} c_0 + (n-2) c_2 + c_4
\end{align*}
For $k \leq 2$, the expression for $a_k$ remains valid as long as
$n \geq k$, under the convention that $c_{2j}=0$ when $2j>n$.

\begin{remark}
Observe that the coefficient of $c_{2j}$ in expression~\eqref{e:a:c}
for $a_k$ is an integer for any $0 \leq j, k \leq \lfloor n/2 \rfloor$.
It equals $1$ when $j=k$ and vanishes when $j>k$.
Taking $F=\mathbb{Q}$, this implies that any polynomial in $W_n$ with integer
coefficients can be represented by a polynomial with integer coefficients,
see also~\cite[Prop.~4.18]{Ladkani23}.
\end{remark}

\section{Coxeter polynomials and characteristic polynomials}
\label{sec:Coxeter}

\subsection{The Coxeter formalism}
We briefly review the Coxeter formalism, which assigns a polynomial to
a square matrix under certain conditions,
see the surveys~\cite{Ladkani23} and~~\cite{LP08}.
Let $F$ be a field, let $n \geq 1$ and let $M_n(F)$ be the space of
square matrices of size $n$ with entries in $F$.
Given a matrix $C \in M_n(F)$, consider the polynomial in $F[x]$
defined by
\[
\phi_C(x) = \det(x \cdot C + C^T) .
\]

\begin{lemma} \label{l:Cox:Wn}
Let $C \in M_n(F)$. Then
\begin{enumerate}[(a)]
\item \label{it:phi:CT}
$\phi_{C^T}(x) = \phi_C(x)$,

\item \label{it:phi:Wn}
$\phi_C(x) \in W_n$.
\end{enumerate}
\begin{proof}
Part~\ref{it:phi:CT} is clear, since $xC^T + C = (xC + C^T)^T$.
For~\ref{it:phi:Wn}, note that
\[
x^n \phi_C(x^{-1}) = x^n \det (x^{-1} C + C^T)
= \det (C + x C^T) = \phi_{C^T}(x)\]
and use~\ref{it:phi:CT}.
\end{proof}

\end{lemma}

\begin{lemma} \label{l:Cox:S}
Let $C \in M_n(F)$.
\begin{enumerate}[(a)]
\item
If there exists $S \in M_n(F)$ such that $C^T = -C S$, then
\[
\phi_C(x) = \det(C) \cdot p_S(x) ,
\]
where $p_S(x)$ denotes the characteristic polynomial of $S$.

\item \label{it:S}
If $C$ is invertible then there exists a unique
$S \in M_n(F)$ such that $C^T = -C S$ and it is given by
$S = - C^{-1} C^T$.
\end{enumerate}
\end{lemma}
\begin{proof}
We can write
$\phi_C(x) = \det (x C - CS) = \det(C) \det(x I_n - S) 
= \det(C) \cdot p_S(x)$.
The other assertion is clear.
\end{proof}

\begin{remark}
The matrix $S$ occurring in part~\ref{it:S} is known as the
\emph{Coxeter transformation} associated with $C$.
In particular, if $\det(C)=1$ then $\phi_C(x) = p_S(x)$ and this
polynomial is known as the \emph{Coxeter polynomial} of $C$.
\end{remark}

We now review the invariance of the Coxeter polynomial under
congruency.
Recall that two matrices $C, C' \in M_n(F)$ are \emph{congruent}
if there exists an invertible matrix $P \in M_n(F)$ such that
$C' = P C P^T$.

\begin{lemma} \label{l:Cox:cong}
Let $C, C' \in M_n(F)$.
\begin{enumerate}[(a)]
\item \label{it:cong:P}
If $C' = P C P^T$ for some $P \in M_n(F)$, then
$\phi_{C'}(x) = (\det P)^2 \cdot \phi_C(x)$.

\item \label{it:cong:C}
If $C$ is invertible, then it is congruent to $(C^T)^{-1}$.

\item
If $C, C'$ are invertible and congruent, then their Coxeter transformations
are similar and in particular they have the same characteristic polynomial.
\end{enumerate}
\end{lemma}
\begin{proof}
\begin{enumerate}[(a)]
\item
This follows from the equalities
\[
\det(xC' + C'^T) = \det(x PCP^T + (PCP^T)^T)
= \det(P) \det(xC + C^T) \det(P^T) .
\]

\item
We can write $C = C (C^T)^{-1} C^T$.

\item
If $C' = P C P^T$, then
\[
-C'^{-1} C'^T =
-(P C P^T)^{-1} (P C P^T)^T = (P^T)^{-1} (-C^{-1} C^T) P^T .
\]
\end{enumerate}
\end{proof}

For matrices with integer entries, the preceding results can be sharpened
as follows.
\begin{lemma} \label{l:Cox:Z}
Let $C, C' \in M_n(\bZ)$.
\begin{enumerate}[(a)]
\item \label{it:CoxZ:cong}
If $C, C'$ are congruent over $\bZ$, then $\phi_C(x) = \phi_{C'}(x)$.

\item \label{it:CoxZ:C}
If $C$ is invertible over $\bZ$, then $C^{-1}, C^T$ are congruent
over $\bZ$ and hence
\[
\phi_{C^{-1}}(x) = \phi_{C^T}(x) = \phi_C(x) .
\]
\end{enumerate}
\end{lemma}
\begin{proof}
Part~\ref{it:CoxZ:cong}
follows from Lemma~\ref{l:Cox:cong}\ref{it:cong:P},
once we note that
any matrix which is invertible over $\bZ$ has determinant equal to
$\pm 1$.
Part~\ref{it:CoxZ:C} follows by combining (the proof of) Lemma~\ref{l:Cox:cong}\ref{it:cong:C}, part~\ref{it:CoxZ:cong}
and Lemma~\ref{l:Cox:Wn}.
\end{proof}

\subsection{Representing Coxeter polynomials by characteristic polynomials}

The Coxeter polynomial of a matrix with determinant $1$ 
is equal to the characteristic polynomial of its Coxeter transformation.
In this section we show that for any square matrix $N$ with $N^2=0$,
the Coxeter polynomial of $I-N$ is represented by the characteristic
polynomial of $N+N^T$.
When $N$ is the adjacency matrix of a bipartite quiver, we recover
the result of~\cite{ACampo76}.

As before, let $F$ be a field and let $n \geq 1$.
We start by recording the following easy observation.
\begin{lemma} \label{l:nilpotent}
If $N \in M_n(F)$ is nilpotent, then so is $a N$ for any
$a \in F$ and moreover, $\det (I_n - a N) = 1$.
\end{lemma}

The key result is summarized in the next proposition.
\begin{prop} \label{p:N1N2}
Let $N_1, N_2 \in M_n(F)$ and assume that $N_1^2 = N_2^2 = 0$.
Then
\begin{equation} \label{e:N1N2}
\det \left( (x^2+1)I_n - x N_1 - x N_2 \right)
= \det \left( (x^2+1) I_n - x^2 N_1 - N_2 \right) .
\end{equation}
\end{prop}
\begin{proof}
We view the matrices whose determinants we would like to evaluate as
matrices with entries in the field $F(x)$ of rational functions
in one variable over $F$. Applying Lemma~\ref{l:nilpotent} for the
matrices $N_1$ and $N_2$ over the field $F(x)$, we deduce that the
determinants $\det(I_n - \alpha N_1)$ and $\det(I_n - \beta N_2)$
are equal to $1$ for any $\alpha, \beta \in F(x)$.

Therefore, the left hand side of~\eqref{e:N1N2} equals the
determinant of any matrix
\[
\left((x^2+1)I_n - x N_1 - x N_2)\right) (I_n - \alpha N_1) =
(x^2+1) I_n - (x + \alpha (x^2+1)) N_1 - x N_2 + x \alpha N_2 N_1
\]
(using $N_1^2=0$), and similarly the right hand side of~\eqref{e:N1N2}
equals the determinant of any matrix
\[
(I_n - \beta N_2) \left( (x^2+1) I_n - x^2 N_1 - N_2 \right) =
(x^2+1) I_n - x^2 N_1 - (1 + \beta (x^2+1)) N_2 + \beta x^2 N_2 N_1 .
\]

By choosing $\alpha = (x^2-x)/(x^2+1)$ and $\beta = (x-1)/(x^2+1)$
we get the desired equality.
\end{proof}

\begin{theorem} \label{t:NN}
Let $N \in M_n(F)$ and assume that $N^2=0$.
Consider the matrices $C = I_n - N$ and $A = N + N^T$
with Coxeter polynomial $\phi_C(x)$ and characteristic polynomial
$p_A(x)$, respectively. Then
\[
\phi_C(x^2) = x^n p_A(x + x^{-1}) .
\]
\end{theorem}
\begin{proof}
By definition,
\[
\phi_C(x^2) = \det(x^2 C + C^T)
= \det \left( (x^2+1) I_n - x^2 N - N^T \right).
\]
On the other hand,
\[
x^n p_A(x + x^{-1}) = x^n \det \left((x+x^{-1})I_n - A \right)
= \det \left((x^2+1) I_n - x N - x N^T \right)
\]
and the result follows by applying Proposition~\ref{p:N1N2} for 
$N_1 = N$ and $N_2 = N^T$.
\end{proof}

\section{Graphs, quivers and their polynomials}
\label{sec:graph}

\subsection{Graphs and their characteristic polynomials}

Let $G$ be a graph on vertex set $\{1, 2, \dots, n\}$,
where loops and multiple edges are allowed at this stage.
Recall that the \emph{adjacency matrix} $A_G$ of $G$ is the square
matrix of size $n$ whose entries $(A_G)_{ij}$ count the number of
edges between vertices $i$ and $j$ in $G$.
For any integer $k \geq 0$, the entry $(A_G^k)_{ij}$ equals the number of
walks of length $k$ from vertex $i$ to vertex $j$ \cite[Prop.~1.3.4]{CRS10}.
In particular, $\tr A_G^k$ equals the number of closed walks of length $k$ in
$G$~\cite[Theorem~3.1.1]{CRS10}.

The classical Newton's identities allow to express the coefficients
of the characteristic polynomial of a square matrix in terms of
its traces. Namely, let $A$ be a square matrix with characteristic
polynomial $p_A(x) = c_0 x^n + c_1 x^{n-1} + \dots + c_n$.
Then $c_0 = 1$ and
\begin{equation} \label{e:Newton}
k c_k = - \sum_{i=1}^{k} c_{k-i} \tr A^i
\end{equation}
for any $1 \leq k \leq n$.
For a combinatorial proof, see~\cite{Zeilberger84}.

Let $G$ be a graph, denote by $p_G(x)$ the characteristic polynomial of
its adjacency matrix $A_G$, and write it as
\[
p_G(x) = x^n + c_1 x^{n-1} + c_2 x^{n-2} + \dots + c_n .
\]
Newton's identities can be used to provide combinatorial interpretations
of the coefficients $c_k$ in terms of walk counts in the graph $G$.
We illustrate this connection
through the following calculations which will be useful in
the subsequent sections.

\begin{lemma} \label{l:G:c124}
Assume that $G$ is simple.
Then $\tr A_G = 0$ and $\tr A_G^2 = 2e$, where $e$ denotes the number of
edges in $G$. Therefore, $c_1 = 0$, $c_2 = -e$ and
\[
c_4 = \frac{1}{2} e^2 - \frac{1}{4} \tr A_G^4 .
\]
\end{lemma}
\begin{proof}
Since $G$ has no loops, $\tr A_G = 0$.
Recall that 
$\tr A_G^2$ equals the number of closed walks of length $2$. As $G$
has no loops and parallel edges, each edge contributes exactly two
such walks. The remaining assertions follow directly from Newton's
identities;
indeed,
$c_1 = -\tr A_G = 0$ while
$2 c_2 = -c_1 \tr A_G - \tr A_G^2 = -2e$, and
\[
4 c_4 = -c_3 \tr A_G - c_2 \tr A_G^2 - c_1 \tr A_G^3 - \tr A_G^4 = 
2e^2 - \tr A_G^4 .
\]
\end{proof}

Let us refine the expression for $c_4$.
\begin{lemma} \label{l:G:c4}
Assume that $G$ is simple. Then
\[
\tr A_G^4 = 2e + 4 \sum_{i=1}^{n} \binom{d_i}{2} + 8q,
\]
where $e$ is the number of edges,
$d_i$ is the degree of vertex $i$, and
$q$ is the number of cycles of length $4$ in $G$.
Therefore, 
\[
c_4 = \binom{e}{2} - \sum_{i=1}^{n} \binom{d_i}{2} - 2q .
\]
\end{lemma}
\begin{proof}
This can be deduced from the proof of~\cite[Theorem~3.1.5]{CRS10}.
A closed walk of length~$4$ traverses one of the subgraphs shown
in Figure~\ref{fig:walk4}.
The numbers of these subgraphs in $G$
are $e$, $\sum_{i=1}^{n} \binom{d_i}{2}$ and $q$, respectively.
For each vertex of these graphs,
Figure~\ref{fig:walk4} shows the number of closed walks of length $4$
starting at that vertex and traversing the graph, so the total
number of closed walks of length $4$ traversing each graph is
$2$, $4$ or $8$, respectively.
\end{proof}

\begin{figure}
\begin{align*}
\begin{array}{c}
\xymatrix{
{\bullet_1} \ar@{-}[r] & {\bullet_1}
}
\end{array}
&&
\begin{array}{c}
\xymatrix@=1pc{
& {\bullet_2} \ar@{-}[dr] \ar@{-}[dl] \\
{\bullet_1} && {\bullet_1}
}
\end{array}
&&
\begin{array}{c}
\xymatrix@=1pc{
& {\bullet_2} \ar@{-}[dr] \ar@{-}[dl] \\
{\bullet_2} \ar@{-}[dr] && {\bullet_2} \ar@{-}[dl] \\
& {\bullet_2}
}
\end{array}
\end{align*}
\caption{The possible subgraphs traversed by closed walks of length $4$.}
\label{fig:walk4}
\end{figure}
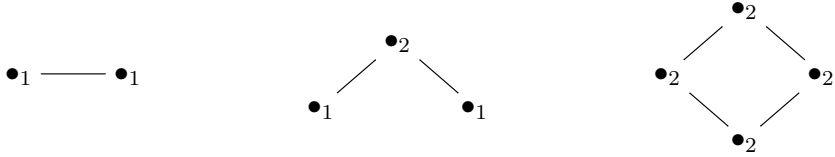

\subsection{Quivers and their Coxeter polynomials}

A quiver $Q$ is a directed graph on vertex set $\{1, 2, \dots, n\}$,
where loops and parallel arrows are allowed at this stage.
Recall that the \emph{adjacency matrix} $N_Q$ of $Q$ is the square matrix
of size $n$ whose entries $(N_Q)_{ij}$ count the number of arrows from
vertex $i$ to vertex $j$ in $Q$.
For any integer $k \geq 0$, the entry $(N_Q^k)_{ij}$ equals the number
of paths of length $k$ from vertex $i$ to vertex $j$.

The underlying graph $G$ of a quiver $Q$ is obtained by forgetting the
orientations of the arrows (i.e.\ replacing each arrow by an edge).
If $Q$ has no loops, then $A_G = N_Q + N_Q^T$.

A quiver is \emph{acyclic} if it does not have oriented cycles of
positive length. In this case $N_Q^n=0$ and in particular $N_Q$ is
nilpotent. Therefore the matrix $C_Q = I_n - N_Q$ is invertible over
$\bZ$ with determinant $1$.

\begin{defn}
The \emph{Coxeter polynomial} $\phi_Q(x)$ of an acyclic quiver $Q$
is defined as the Coxeter polynomial of the matrix $C_Q$.
By Lemma~\ref{l:Cox:S}, $\phi_Q(x) = \det(x I_n + C_Q^{-1} C_Q^T)$.
\end{defn}

If $Q$ is acyclic with adjacency matrix $N_Q$, then the $(i,j)$-entry
in the matrix $I_n + N_Q + N_Q^2 + \dots + N_Q^{n-1}$ counts the
total number of oriented paths starting at vertex $i$ and ending at $j$.
This matrix is known as the \emph{Cartan matrix} of the path algebra
of $Q$, and its inverse is the matrix
$C_Q = I_n - N_Q$, which is related to the Euler form of that algebra.
Lemma~\ref{l:Cox:Z} now implies that the
Coxeter polynomial $\phi_Q(x)$ may alternatively be defined as
the Coxeter polynomial of the Cartan matrix corresponding to $Q$.

\begin{defn}
A quiver is \emph{bipartite} if its vertex set can be partitioned into
two disjoint subsets $V' \sqcup V''$ such that all arrows start at $V'$
and end at $V''$. 
\end{defn}

The relationship between the Coxeter polynomial of a bipartite quiver
and the characteristic polynomial of its underlying graph was discovered
by A'Campo, and is formalized in the following proposition.

\begin{prop}[\cite{ACampo76}] \label{p:Q:G}
Let $Q$ be a bipartite quiver with underlying graph $G$. Then
\[
\phi_Q(x^2) = x^n p_G(x + x^{-1}) .
\]
\end{prop}
\begin{proof}
$N_Q^2=0$ since $Q$ is bipartite.
Now apply Theorem~\ref{t:NN} for the matrix $N_Q$.
\end{proof}

A vertex $s$ in a quiver is called a \emph{sink} if there are no arrows
starting at $s$. Similarly, $s$ is called a \emph{source} if there are
no arrows ending at $s$. Thus, a quiver is bipartite if every vertex is
a sink or a source.
If $s$ is a sink or a source in a quiver $Q$, the
\emph{reflection} at $s$ is the quiver $Q'$ obtained by reversing 
the direction of all arrows incident to $s$, see~\cite{BGP73}.
It is well-known that if $Q$ is acyclic, then so is $Q'$, and the
matrices $C_Q$ and $C_{Q'}$ are congruent over $\bZ$.
While this fact is a consequence of the derived equivalence between the path
algebras of these quivers~\cite{Happel88} and the invariance of the Euler
form under derived equivalences, it can also be verified directly.
Consequently, the Coxeter polynomials $\phi_Q(x)$ and $\phi_{Q'}(x)$
coincide.

Any two orientations of a tree $T$ are related by a sequence of 
reflections~\cite{BGP73}. As reflections preserve the Coxeter polynomial,
it follows that all orientations of $T$ share the same Coxeter polynomial,
which we denote by $\phi_T(x)$.

Finally, recall that two arrows in a quiver are \emph{parallel}
if they start at the same vertex and end at the same vertex.
The adjacency matrix of a quiver without parallel arrows
has its entries in $\{0, 1\}$.

\section{Proofs of the main results}
\label{sec:proofs}

\subsection{Proof of Theorem~\protect{\ref{t:Q}}}

If $Q$ is a bipartite quiver on $n$ vertices with underlying graph $G$,
then $\phi_Q(x^2) = x^n p_G(x+x^{-1})$ by Proposition~\ref{p:Q:G},
so we can write
\[
p_G(x) = x^n + c_2 x^{n-2} + c_4 x^{n-4} + \dots
\]
by Proposition~\ref{p:represent}.
Our hypotheses on $Q$ imply that the graph $G$ is simple with $q$
cycles of length $4$, hence by Lemma~\ref{l:G:c124} and Lemma~\ref{l:G:c4},
\begin{align*}
c_2 = -e &,&
c_4 = \frac{e(e-1)}{2} - 2q - \sum_{i=1}^{n} \binom{d_i}{2} ,
\end{align*}
where $e$ is the number of arrows of $Q$ and $d_i$ is the degree of
vertex $i$.

Write now
\[
\phi_Q(x) = x^n + a_1 x^{n-1} + a_2 x^{n-2} + \dots
\]
By Proposition~\ref{p:a:c},
$a_1 = n + c_2 = n - e$ and
\begin{align*}
a_2 &= \binom{n}{2} + (n-2) c_2 + c_4 \\
&= \frac{n(n-1)}{2} + (n-2) (-e) + \frac{e(e-1)}{2} - 2q
- \sum_{i=1}^{n} \binom{d_i}{2} \\
&= \frac{1}{2}
\left( n(n-1) - (2n-3)e + e^2 \right) -2q
- \sum_{i=1}^{n} \binom{d_i}{2} \\
&= \frac{(n-e)(n-1-e)}{2} + e - 2q - \sum_{i=1}^{n} \binom{d_i}{2} ,
\end{align*} 
completing the proof of Theorem~\ref{t:Q}.

To obtain the expression~\eqref{e:Q:a2v} for $a_2$,
start by rewriting the sum as
\[
\sum_{i=1}^{n} \binom{d_i}{2} = 
\sum_{i=1}^{n} \frac{d_i(d_i-1)}{2} =
\sum_{i=1}^{n} \frac{(d_i-1)(d_i-2)}{2} + 2e - n
\]
and observe that if $d_i=0$, the term $(d_i-1)(d_i-2)/2$
contributes $1$ to the sum, while if $d_i>0$, it equals the binomial
coefficient $\binom{d_i-1}{2}$. Therefore
\begin{align*}
a_2 &= \frac{1}{2}
\left( n(n-1) - (2n-3)e + e^2 \right)
- 2q - 2e + n - \sum_{i=1}^n \frac{(d_i-1)(d_i-2)}{2} \\
&= \frac{1}{2} \left ( n(n+1) - (2n+1)e + e^2 \right)
- 2q - v - \sum_i \binom{d_i-1}{2} \\
&= \frac{(n-e)(n+1-e)}{2} - v - 2q - \sum_i \binom{d_i-1}{2} ,
\end{align*}
where $v$ is the number of isolated vertices in $Q$ 
and the sum runs over the non-isolated vertices.

\subsection{Proof of Corollary~\protect{\ref{c:T}}}

Any tree admits a bipartite orientation.
To get Corollary~\ref{c:T}, use the expression~\eqref{e:Q:a2v} and note
that for a tree with $n$ vertices, $e=n-1$ and there are no isolated vertices.

\bibliographystyle{amsplain}
\bibliography{coeff}

\end{document}